\documentclass{amsart}	
\usepackage{amsfonts}
\usepackage{amsmath,amsthm}
\usepackage{amsfonts,amssymb}
\usepackage{mathrsfs}
\usepackage{enumerate}

\newtheorem{thm}{Theorem}[section]
\newtheorem{cor}[thm]{Corollary}
\newtheorem{lem}[thm]{Lemma}

\newtheorem{defn}[thm]{Definition}

\theoremstyle{remark}
\newtheorem{rem}[thm]{Remark}

\numberwithin{equation}{section}

\newcommand{\xqedhere}[2]{%
  \rlap{\hbox to#1{\hfil\llap{\ensuremath{#2}}}}}

\def\({\Bigl(}
\def \){ \Bigr)}

\newcommand{\norm}[1]{\left\Vert#1\right\Vert}
\newcommand{\abs}[1]{\left\vert#1\right\vert}
\newcommand{\set}[1]{\left\{#1\right\}}

\begin{document}

\title[] {Dunkl translations, 
Dunkl--type $BMO$ space 
and\\
Riesz transforms for Dunkl transform on $L^\infty$ }

\author[]{Wentao Teng} 

\address{School of Science and Technology, Kwansei Gakuin University, Japan.}

\email{wentaoteng6@sina.com.}


\begin{abstract} In this paper, we will give some results on the support of Dunkl translations on compactly supported functions. Then we will define Dunkl--type $BMO$ space and Riesz transforms for Dunkl transform on $L^\infty$, and prove the boundedness of Riesz transforms from $L^\infty$ to Dunkl--type $BMO$ space under the uniform boundedness assumption of Dunkl translations. The proof and the definition in Dunkl setting will be harder than in the classical case for the lack of some similar properties of Dunkl translations to that of classical translations. We will also extend the preciseness of the description of support of Dunkl translations on characteristic functions by Gallardo and Rejeb to that on all nonnegative radial functions in $L^2(m_k)$.

\end{abstract}

\maketitle
\input amssym.def

\section{Introduction}

Let $T$ be bounded operator on $L^2(\mathbb{R}^N)$ 
and $K$ be a function on
$\mathbb{R}^N\times\mathbb{R}^N\backslash$ $\left\{\left(x,x\right):\;x\in\mathbb{R}^N\right\}$, such that for any $f\in
L^2 (\mathbb{R}^N)$ with a compact support, 
$$Tf(x)=\int_{\mathbb{R}^N}K(x,y)f(y)dy,\;x\in\mathbb{R}^N\backslash supp(f),$$
where $K$ satisfies
\begin{equation}\int_{\left|x-y\right|>2\left|x-w\right|}\left|K(x,y)-K(w,y)\right|dy\leq C,\end{equation}
then $T$ is a bounded operator from $L^\infty$ to $BMO$ space, or
the space of bounded mean oscillation functions. Let
$K(x,y)=c_N(x_j-y_j)/\vert x-y\vert^{N+1},\;j=1,\:\dots,\;N$. For
any $\varepsilon>0$ consider the truncation $K_\varepsilon$
defined by $K_\varepsilon(x,y)=K(x,y)$ if $\vert
x-y\vert>\varepsilon$, and $K_\varepsilon(x,y)=0$ if $\vert
x-y\vert\leq\varepsilon$. If $f$ is a bounded function, the
ordinary Riesz transform is defined by
$$R_j(f)(x)=\lim_{\varepsilon\rightarrow0}\int_{\mathbb{R}^N}(K_\varepsilon(x,y)-K_1(0,y))f(y)dy.$$
It is well known the Riesz transform is bounded on
$L^2(\mathbb{R}^N)$ and that the kernel $K(x,y)$ satisfies (1.1), and so is a bounded operator from $L^\infty$ to $BMO$
space.  In this paper we will extend analogous results to the
context of Dunkl theory. 

In \cite{AS}, the $L^p$-boundedness, $1<p<\infty$ and weak $L^1$ boundedness of Riesz transforms for Dunkl transform was proved by adapting the classical $L^p$-theory of Calder\'on--Zygmund, and 
so 
the Riesz transforms can be defined as bounded operators on $L^p,\;1<p<\infty$ and weakly bounded operators on $L^1$ (see also \cite{I} for the $L^p$ boundedness of Dunkl--Riesz transforms with radial power weights). But there is no reasonable and coincident definition of Riesz transforms by integral on $L^1(m_k)$ in Dunkl setting. Recently, the Riesz transforms were defined in a weak sense on 
$L^1(m_k)$ (see \cite{ADH}) using a test function space containing a Poisson kernel, and it 
was shown in \cite{ADH} and \cite{DH} that in Dunkl setting, the real Hardy space $H_\triangle^1$ associated to the Dunkl Laplacian $\triangle$ can be characterized by Riesz transforms and also coincide with 
$H_{atom}^1$. 

The formula that Dunkl translation operators $\tau_x$ are contractions on $L^2(m_k)$ is well-known:
$${\left\|\tau_yf\right\|}_{2,\;k}\leq{\left\|f\right\|}_{2,\;k},\;f\in L^2(m_k).$$
Assume the uniform $L^1$-boundedness of the Dunkl translations (see \cite{DH1}). Then by Riesz-T\"orin interpolation and skew-symmetry of Dunkl translations, the 
uniform $L^p$-boundedness ($1\leq p\leq\infty$) can be get immediately, that is, for any root system $R$ and multiplicity function $k\geq0$ and for any $f\in L^p(m_k)$,
$${\left\|\tau_yf\right\|}_{p,\;k}\leq C{\left\|f\right\|}_{p,\;k},$$
where $C$ is a constant independent of $y$. It has been known that this assumption holds for radial functions and one-dimensional case, and hence for $G=Z_2^N$ case. 
Here we call this assumption as the uniform boundedness assumption of Dunkl translations. There have been many results based on this long open assumption and it would be an excellent work if one could prove this assumption. In \cite{Go}, the authors proved the uniform boundedness of the spherical average of Dunkl translations, and as applications, they found this uniform boundedness assumption can be avoided in the proof of some related results. But it is still inevitable for many other results, such as the $L^p$ boundedness of the Dunkl multiplier operator for $s>\mathbf N/2$ (see \cite[Theorem 8.1]{DH1}), and the Theorem 1.1 in this paper as well. 

In this paper we will define Dunkl--type $BMO$ space and $\mathcal R_jf$, where $\mathcal R_j$ is the Riesz transform for Dunkl transforms, as Dunkl--type $BMO$ functions for all $f\in L^\infty$. Then under the uniform boundedness assumption of Dunkl translations, we 
will prove the boundedness of the Riesz transforms from $L^\infty$ to Dunkl--type $BMO$ space. This will also mean a half part of duality of the Hardy space $H_\triangle^1$ and Dunkl--type $BMO$ space.

\begin{thm} Under the uniform boundedness assumption of Dunkl translations, the Riesz transform for Dunkl transforms are bounded operators from $L^\infty$ to the Dunkl--type $BMO$ space.
\end{thm}

The part ii of the following theorem shows that the support of $\tau_{-x}f$ obtained in \cite[Theorem 1.7]{DH1} (part i of the following Theorem) is precise when the multiplicity function $k>0$. The preciseness has been proved for characteristic functions by Gallardo and Rejeb \cite{Ga} and we extend the result to any nonnegative radial functions on $L^2(m_k)$ in this paper.
\begin {thm}
If $f\in L^2(m_k)$ and $suppf\subseteq B(0,r)$, then for any $x\in\mathbb{R}^N$\\
i).(See \cite[Theorem 1.7]{DH1})\;$$supp\tau_xf(-\cdot)\subseteq\bigcup_{g\in G}B(gx,\;r).$$
ii).\;If the multiplicity function $k>0$ and let $f$ be a nonnegative radial function on $L^2(m_k)$,  $suppf=B(0,r)$, then 
$$supp\tau_xf(-\cdot)=\bigcup_{g\in G}B(gx,\;r).$$
\end {thm}
The part i of this theorem also means that the Dunkl translation of a function on $L^2(m_k)$ with a compact support is compactly supported. This will be used in the proof of Theorem 1.1. However, different from classical analysis, $suppf\subseteq\bigcup_{g\in G}B(gx,r)$, $x\in\mathbb{R}^N$, can not usually imply $supp\tau_xf\subseteq B(0,r)$ as will be shown in Section 2. This led to the differences between the proof of the boundedness of Riesz transforms from $L^\infty$ to the $BMO$ space in Dunkl setting and that in classical case.

This paper is organized as follows. In Section 2 we present some definitions and fundamental results from Dunkl's analysis. In Section 3, we will prove Theorem 1.1. ii) and give more information about the support of Dunkl translations on compactly supported functions based on the results of  \cite{DH1}. Section 4 is devoted to Riesz transforms for Dunkl transform. In Section 5, the Dunkl--type $BMO$ space and Riesz transforms for Dunkl transform on $L^\infty$ will be defined and we will prove the boundedness of the Riesz transforms from $L^\infty$ to Dunkl--type $BMO$ space. We first prove for compactly supported functions and then for all functions on $L^\infty$ using a Lemma we will give in the section.

\section{Preliminaries}

For any $x,y$ in the Euclidean space $\mathbb{R}^N$, denote by $\left\langle x,y\right\rangle=\sum\nolimits_{j=1}^Nx_jy_j$ the standard inner product associated with 
norm $\left\|x\right\|$. For any nonzero
vector $\alpha\in\mathbb{R}^N$, define the reflection $\sigma_\alpha$
with respect to the hyperplane $\alpha^\perp$ orthogonal to $\alpha$, 
$$\sigma_\alpha(x)=x-2\frac{\left\langle x,\alpha\right\rangle}{\left\|\alpha\right\|^2}\alpha.$$ 
A finite set $R\subset\mathbb{R}^N\backslash\left\{0\right\}$ is
called a $root$ $system$ if $\sigma_\alpha(R)=R$ for any $\alpha\in R$.
Given a root system $R$, the finite subgroup $G$ of $O(N)$ generated by the
reflections $\sigma_\alpha$ is called the $finite$ $relection$
$group$ of the root system. Define a $multiplicity$ $function$ $k:R\rightarrow \mathbb C$
such that $k$ is $G$-invariant, that is, $k\left(\alpha\right)=k\left(\beta\right)$ if $\sigma_\alpha$ and $\sigma_\beta$ are conjugate. We assume $k\geqslant0$ in 
this paper. The $Dunkl\;operators\;T_\xi$,
$\xi\in\mathbb{R}^N$, which were introduced in \cite{Du}, are defined by the following deformations by difference operators of directional derivatives $\partial_\xi$: \begin{align*}T_\xi f(x)&=\partial_\xi f(x)+\sum_{\alpha\in
R}\frac{k(\alpha)}2\left\langle\alpha,\;\xi\right\rangle\frac{f
(x)-f(\sigma_\alpha(x))}{\left\langle\alpha,\;x\right\rangle}\\&=\partial_\xi
f(x)+\sum_{\alpha\in
R^+}k(\alpha)\left\langle\alpha,\;\xi\right\rangle\frac{f(x)-f
(\sigma_\alpha(x))}{\left\langle\alpha,\;x\right\rangle},\end{align*}
where $R^+$ is any fixed positive subsystem of $R$. They commute pairwise and are skew-symmetric with respect to the $G$-invariant measure $dm_k(x)=h_k^2(x)dx$, where
$$h_k(x)=\prod_{\alpha\in R^+}\vert\left\langle\alpha,\;x\right\rangle\vert^{k(\alpha)}$$
and $m_k$ is a doubling measure, that is, there is a constant $C >
0$ such that
\begin{equation}\label{mk}m_k(B(x,2r))\leq Cm_k(B(x,r))\end{equation} for $x\in\mathbb{R}^N,\;r>0$, where $B(x,r)=\{y\in\mathbb{R}^N:$
$\left\|x-y\right\|\leq r\}.$ Denote by $\mathbf N\boldsymbol=N+\sum_{\alpha\in R}k(\alpha)$ the homogeneous dimension of the root system.
Let $e_j,\;j=1,2,...,N,$ be the canonical orthonormal basis in $\mathbb{R}^N$ and denote $T_j=T_{e_j}.$
The $Dunkl$ $Laplacian$ is defined by $\triangle={\textstyle\sum_{j=1}^N}T_j^2$. It commutes with the action of $G$, that is, $g\circ\triangle=\triangle\circ g$ for 
any $g\in G$, and has the following explicit expression,
$$\triangle f\left(x\right)=\triangle_{eucl}f\left(x\right)+2\sum_{\alpha\in R^+}k\left(\alpha\right)\left(\frac{\left\langle\nabla_{eucl}f,\alpha\right\rangle}{\left
\langle\alpha,x\right\rangle}-\frac{f\left(x\right)-f\left(\sigma_\alpha\left(x\right)\right)}{\left\langle\alpha,x\right\rangle^2}\right).$$
The operator $-\triangle$ is essentially self-adjoint and positive definite and so $\triangle$ is the generator of the contraction semigroup ${\left\{e^{t\triangle}\right\}}_{t\geq0}$.

The operators $\partial_\xi$ and $T_\xi$ are intertwined by a Laplace--type operator $$V_kf(x)=\int_{\mathbb{R}^N}f(y)d\mu_x(y),$$ associated to a family of probability 
measures $\left\{\mu_x\vert\;x\in\mathbb{R}^N\right\}$ with compact support, that is, 
$$T_\xi\circ V_k=V_k\circ \partial_\xi.$$ 
Specifically, the support of $\mu_x$ is contained in the convex hull $co(G\cdot x)$, where $G\cdot x=\left\{g\cdot x\vert\;g\in G\right\}$ is the orbit of $x$.
For any Borel set $B$ and any $r>0$, $g\in G$, the probability measures satisfy 
$$\mu_{rx}\left(B\right)=\mu_x\left(r^{-1}B\right),\;\mu_{gx}\left(B\right)=\mu_x\left(g^{-1}B\right).$$

The Dunkl kernel $E(x,y)$ is defined by 
$$E\left(x,y\right)=V_k\left(e^{\left\langle\cdot,y\right\rangle}\right)\left(x\right)=\int_{\mathbb{R}^d}e^{\left\langle\eta,y\right\rangle}d\mu_x\left(\eta\right).$$
It is the generalization of exponential function $e^{\left\langle x,y\right\rangle}$.
For any fixed $y\in\mathbb{R}^N$, the Dunkl kernel $E(x,y)$ is the unique analytic solution to the differential equation system
$$T_\xi f=\left\langle\xi,y\right\rangle f,\;\;f(0)=1.$$

For $f\in L^1(m_k)$ the Dunkl transform is defined by 
$$F(f)(\xi)=\frac1{c_k}\int_{\mathbb{R}^N}f(x)E(-i\xi,\;x)dm_k(x),\;c_k=\int_{\mathbb{R}^N}e^{-\frac{\left|x\right|^2}2}dm_k(x).$$

Obviously, $F\left(\triangle f\right)\left(\xi\right)=-\left\|\xi\right\|^2Ff\left(\xi\right)$ and $F(e^{t\triangle}f)=e^{t\left|\cdot\right|^2}F\left(f\right),\;f\in 
L^2\left(m_k\right).$
It follows that 
$$e^{t\triangle}f\left(x\right)=k_t\ast f=\int_{\mathbb{R}^N}h_t(x,y)f(y)dm_k(y),$$
where $k_t(x)=c_k^{-1}\left(2t\right)^{-\mathbf N/2}e^{-\left|x\right|^2/\left(4t\right)}$ and the heat kernel $h_t(x,y)=\tau_xk_t(-y).$

Let $x\in\mathbb{R}^N$, the Dunkl translation operator $\tau_x$ is defined on $L^1(m_k)$ by,
$$F(\tau_x(f))(y)=E\left(ix,\;y\right)Ff(y),\;y\in\mathbb{R}^N.$$
It can also be defined by 
$$\tau_xf\left(y\right)={\left(V_k\right)}_y{\left(V_k\right)}_x\left[\left(V_k\right)^{-1}\left(f\right)\left(x+y\right)\right].$$
Here are some basic properties of Dunkl translatons.
$$\begin{array}{l}1.\;(identity)\;\tau_0=I;\\2.\;(Symmetry)\;\tau_xf(y)=\tau_yf(x),\;x,\;y\in\mathbb{R}^N,\;f\in S(\mathbb{R}^N);\\3.\;(Scaling)\;\tau_x(f_\lambda)=(\tau_{\lambda^{-1}x}f)_\lambda,\;\lambda>0,\;x\in\mathbb{R}^N,\;f\in S(\mathbb{R}^N);\;\\4.\;(Commutativity)\;T_\xi(\tau_xf)=\tau_x(T_\xi f),\;x,\;\xi\in\mathbb{R}^N;\\5.\;(Skew-symmetry)\\\;\;\;\;\displaystyle{\int_{\mathbb{R}^N}\tau_xf(y)g(y)dm_k(y)}=\displaystyle{\int_{\mathbb{R}^N}f(y)\tau_{-x}g(y)dm_k(y)},\;x\in\mathbb{R}^N,\;f,\;g\in S(\mathbb{R}^N);\end{array}$$
The Dunkl translations can be defined on $L^p(m_k),\;1\leq p\leq\infty$ in the distributional sense due to the latter formula. Further, 
$$\int_{\mathbb{R}^N}\tau_xf(y)dm_k(y)=\int_{\mathbb{R}^N}f(y)dm_k(y),\;x\in\mathbb{R}^N,\;f\in S(\mathbb{R}^N).$$

The following formula for radial functions was first proved by R\"osler \cite {R} for Schwartz functions,
and was then extended to all continuous radial functions in \cite{DW}:
\begin {equation}
\tau_xf(-y)=\int_{\mathbb{R}^N}(\widetilde f\circ A)(x,\;y,\;\eta)d\mu_x(\eta),\;x,\;y\in\mathbb{R}^N,
\end {equation}
where $f(x)=\widetilde f(\left\|x\right\|)$ and $$A(x,\;y,\;\eta)=\sqrt{\left\|x\right\|^2+\left\|y\right\|^2-2\left\langle y,\;\eta\right\rangle}=\sqrt{\left\|x\right\|^2-\left\|\eta\right\|^2+\left\|y-\eta\right\|^2},$$
It follows from the symmetry of Dunkl translations that (see \cite{Ga1})
$$\tau_{-x}f(y)=\tau_yf(-x)=\tau_xf(-y),\;x,\;y\in\mathbb{R}^N,\;f\in S(\mathbb{R}^N)_{rad}.$$

The Dunkl convolution of Schwartz functions is defined by 
$$(f\ast g)(x)=\int_{\mathbb{R}^N}f(y)\tau_xg(-y)dm_k(y),$$
or can be written as 
$$(f\ast g)(x)=\int_{\mathbb{R}^N}(Ff)(\xi)(Fg)(\xi)E(ix,\;\xi)dm_k(\xi).$$
The following are some basic properties of Dunkl convolution,\\
$1.\;F(f\ast g)=Ff\cdot Fg;$\\
$2.\;F(f\cdot g)=Ff\ast Fg;$\\
$3.\;f\ast g=g\ast f;$\\
$4.\;(f\ast g)\ast h=f\ast(g\ast h);$\\
$5.\;{\left\|f\ast g\right\|}_{2,\;k}\leq{\left\|f\right\|}_{1,\;k}{\left\|g\right\|}_{2,\;k},\;f\in L^1(m_k),\;g\in L^2(m_k).$\\

\section{Some results on the supports of Dunkl translations}

\noindent{\it Proof of Theorem 1.2.\;ii)}. It suffices to prove that
$$supp\tau_xf(-\cdot)\supseteq\bigcup_{g\in G}B(gx,\;r).$$

Firstly, we will prove for continuous nonnegative radial functions. Suppose there exists a $y\in\bigcup\limits_{g\in G}B(gx,\;r)$,
that is, there exists a $g\in G$, $\left\|y-g\cdot x\right\|\leq
r$, such that $y\not\in supp\tau_xf \left(-\cdot\right)$, that is,
there exists $\varepsilon>0$, for any $z\in B(y,\varepsilon)$,
$$0=\tau_xf(-z)=\int_{\mathbb{R}^N}\widetilde f(\sqrt{\left\|x\right\|^2+\left\|z\right\|^2-2\left\langle z,\;\eta\right\rangle})d\mu_x(\eta),$$
then
$$\widetilde f(\sqrt{\left\|x\right\|^2+\left\|z\right\|^2-2\left\langle z,\;\eta\right\rangle})=0,\;for\;any\;\eta\in supp\mu_x.$$
By a result of Gallardo and Rejeb (see \cite{Ga}), that the orbit
of $x$, $G\cdot x$, is contained in the support of $\mu_x$ if
$k>0$, for the above $g$ we can select $\eta=g\cdot x$, then
$f(z-g\cdot x)=\widetilde f(\left\|z-g\cdot x\right\|)=0$. For any
$z_1\in B(y-g\cdot x,\varepsilon)$, $z_1+g\cdot x\in
B(y,\varepsilon)$, and so $f\left(z_1\right)=f(z_1+g\cdot x-g\cdot
x)=0$, which means $y-g\cdot x\not\in suppf$, and this leads to a
contradiction to that $suppf=B(0,r)$.

Then for any nonnegative radial functions $f$ on $L^2(m_k)$, $suppf=B(0,r)$, by the density of continuous functions with compact support $B(0,r)$ in $L^2(B(0,r),m_k)$, there exists a sequence of continuous nonnegative radial functions $g_n$ whose support is $B(0,r)$, such that $f/2$ can be approximated by $g_n$ with respect to $L^2$-norm. So for any nonnegative smooth function $\varphi$ on $\mathbb{R}^N$ with compact support, $\int g_n\varphi\rightarrow\int\frac f2\varphi$. If $(supp\varphi)^\circ\cap B(0,r)\neq\varnothing$, then $\int f\varphi>0$, where $A^\circ$ stands for the interior of $A$ for any $A\subseteq\mathbb R^N$. So there exists a sufficiently large natural number $L$ such that $\int g_L\varphi<\int f\varphi$. If $(supp\varphi)^\circ\cap B(0,r)=\varnothing$, then for any $n\in\mathbb{N}$, $\int g_n\varphi=\int f\varphi=0$.  So for any nonnegative smooth function $\varphi$ on $\mathbb{R}^N$ with compact support, $\int g_L\varphi\leq\int f\varphi$. Thus $g_L\leq f\;a.e.$ and $\int\tau_{-x}g_L\cdot\varphi\leq\int\tau_{-x}f\cdot\varphi$ by positivity of Dunkl translations on radial functions. Let $D=(supp\tau_{-x}f)^c$, then $D$ is the largest open set such that $0=\int\tau_{-x}f\cdot\varphi$ for any smooth functions functions $\varphi$ with compact support in $D$. If $\varphi\geq0$, then $\int\tau_{-x}g_L\cdot\varphi=0$. Then by $\tau_{-x}g_L\geq0$, \begin{align*}\bigcup_{g\in G}B(gx,\;r)=supp\tau_{-x}g_L\subseteq D^c=supp\tau_{-x}f.\xqedhere{2.814cm}{\qed}\end{align*}

\begin {rem}
This theorem does not hold for $k\geqslant0$. For example, for any nontrival finite reflection group $G$, we can take $k=0$. Then $supp\tau_xf(-\cdot)=B(x,r)$ when $suppf=B(0,r)$ and is obviously not $\bigcup_{g\in G}B(gx,\;r)$ since $G$ is nontrival. We refer to \cite[Example 3.1]{Ga} for more counterexamples.
\end {rem}

\begin {cor}
Let $f\in L^2(m_k)$ and $x\in\mathbb{R}^N$, $suppf\cap\bigcup\limits_{g\in
G}B(gx,r)=\varnothing$, then 
$supp\tau_xf\cap B(0,r)=\varnothing.$
\end {cor}
\begin {proof}
For any function $g\in L^2(m_k)$, $suppg\subseteq B(0,r)$, from Theorem 1.2. i),
$$supp\tau_{-x}g\subseteq\bigcup_{g\in G}B(gx,r).$$
By the skew-symmetry of Dunkl translations,
$$\int_{\mathbb{R}^N}\tau_xf(y)g(y)dm_k(y)=\int_{\mathbb{R}^N}f(y)\tau_{-x}g(y)dm_k(y)=0.$$
Then $\tau_xf(y)=0$, $y\in B(0,r).$
\end {proof}

Define the distance of the orbits $G\cdot x$ and $G\cdot y$ (see
\cite{DH}),
\begin{equation}\\d_G(x,y)=\underset{g\in
G}{min}\vert\vert g\cdot y-x\vert\vert.\end{equation} For any
fixed point $x$ and a ball $B(x,r)$ with center $x$, let
$B^\ast=B(x,2r)$ and $Q^\ast=\bigcup_{g\in G}gB^\ast$. For any
$y\in B(x,r)$, if $z\in\mathbb{R}^N\backslash Q^\ast$, then (see
\cite{AS})
\begin{equation}
d_G(x,z)>2\vert\vert y-x\vert\vert\\.
\end{equation}

\begin {thm}
Let $f\in L^p(m_k),\;1\leq p<\infty$ be a radial function,
$suppf\cap B(0,r)=\varnothing,$ then for any $x\in\mathbb{R}^N$,
\begin {equation}
supp\tau_xf(-\cdot)\cap\bigcap_{g\in G}B(gx,r)=\varnothing.
\end {equation}

\end {thm}
\begin {proof}
Let us prove for continuous radial functions first. It is easy to
see that 
\begin {equation}\underset{g\in G}{max}\left\|g\cdot x-y\right\|\geq
A\left(x,\;y,\;\eta\right)\geq d_G(x,\;y)\end {equation}
for any
$x,\;y\in\mathbb{R}^N$ and $\eta\in co(G\cdot x)$. For any
continuous radial functions $f$ with support contained in
$B(0,r)^c$, if
$$\tau_xf(-y)=\int_{\mathbb{R}^N}(\widetilde f\circ A)\left(x,\;y,\;\eta\right)d\mu_x(\eta)\neq0,$$
then $\underset{g\in G}{max}\left\|g\cdot x-y\right\|\geq r$.
Therefore, $supp\tau_xf(-\cdot)\cap\bigcap\limits_{g\in
G}B(gx,\;r)=\varnothing$. By the density of continuous functions
on $L^p(m_k)$ and the continuity of Dunkl translations on
$L^p(m_k)$ for radial functions, (3.2) can be extended to any
radial functions in $L^p(m_k)$.
\end {proof}
\begin {rem}
One may expect that $supp\tau_xf(-\cdot)\cap\bigcup_{g\in G}B(gx,r)=\varnothing$, but this is not correct even for a characteristic function for a general finite reflection group $G$. From a similar argument as in Corollay 3.2, this also means that $suppf\subseteq\bigcup_{g\in G}B(gx,r)$ can not usually imply $supp\tau_xf\subseteq B(0,r)$ as in classical case. 
 \end {rem}
As an immediate consequence of the theorem, the condition of the
Corollary 4.1 in \cite{DH1} can be weakened for radial functions.
\begin {cor} Suppose for all $g\in G$ and $x,\;y\in\mathbb{R}^N$,
$\left\|g\cdot x-y\right\|<1$. Let $f$ be a radial function in
$L^p(m_k),\;1\leq p<\infty$, $f(z)=0$ for all $z\in B(0,1)$,
then $ \tau_xf(y)=0$. \end {cor} 

\begin {rem}

This theorem cannot be extended to functions not necessarily radial because the Stone-Weierstrass theorem does not hold on $B(0,r)^c$, and there is no more precise result on the support of the distribution associated to  $\tau_xf(y)$ other than \cite{AA}.
 \end {rem}

\section{Riesz transforms for Dunkl transform}

The Riesz transforms ${\mathcal R}_j\;$in the Dunkl setting are defined by
$${\mathcal R}_j(f)(x)=\lim_{\varepsilon\rightarrow0}d_k\int_{\left\|y\right\|\geq\varepsilon}\tau_xf(-y)\frac{y_j}{\left\|y\right\|^{p_k}}dm_k(y),\;f\in S(\mathbb{R}^N),$$
where $j=1,\cdots,N$ and $d_k=2^{\mathbf N/2}\Gamma((\mathbf N+1)/2)/\sqrt\pi,\;$ $p_k=\mathbf N+1$.
It has been proved in \cite{TX} that  
\begin{align*}F(\mathcal R_jf)(\xi)=-i\frac{\xi_j}{\left\|\xi\right\|}(Ff)(\xi),\;j=1,2,\cdots,n \end{align*}
and $\mathcal R_ j$ is a bounded operator on $L^2(m_k)$.
Clearly,
\begin{align*}\mathcal R_jf=-T_{e_j}(-\triangle)^{-1/2}f=-\lim_{\varepsilon\rightarrow0,\;M\rightarrow\infty}c\int_\varepsilon^MT_{e_j}e^{t\triangle}f\frac{dt}{\sqrt t},\end{align*}
and the integral converges for $f\in L^2(m_k)$. It is obvious that the Riesz transforms commute with the Dunkl translations. If $f\in L^2(m_k)$ and has a compact support, it was shown in  \cite{AS} that for all
$x\in\mathbb{R}^N$ such that $g\cdot x\in\mathbb{R}^N\backslash
supp(f)$ for any $g\in G$, 
$$\mathcal R_j(f)(x)=\int_{\mathbb{R}^N}\mathcal K_j(x,y)f(y)dm_k(y),$$
where\\
$\mathcal K_j(x,y)=d_k\left\{\mathcal K_j^{(1)}(x,y)+\displaystyle\sum\limits_{\alpha\in R^+}\frac{k\left(\alpha\right)\alpha_j}{p_k-2}\mathcal K_j^{(\alpha)}(x,y)\right\},$\\

$\mathcal K_j^{(1)}(x,y)=\displaystyle\int_{\mathbb{R}^N}\frac{\eta_j-y_j}{A^{p_k}(x,y,\eta)}d\mu_x(\eta),$\\

$\mathcal K_j^{(\alpha)}(x,y)=\displaystyle\frac1{\left\langle
y,\alpha\right\rangle}\int_{\mathbb{R}^N}\left[\frac1{A^{p_k-2}\left(x,y,\eta\right)}-\frac1{A^{p_k-2}\left(x,
\sigma_\alpha
\left(y\right),\eta\right)}\right]d\mu_x(\eta),\;\alpha\in R^+$,\\
and $\mathcal K_j(x,y)$ satisfies the condition
\begin{equation}
\int_{d_G(x,z)>2\left\|y-x\right\|}\left|\mathcal K_j(z,\;x)-\mathcal K_j(z,\;y)\right|dm_k(z)\leq
C.\end{equation}
And the authors proved that $\mathcal R_j$ is a bounded operator on $L^p(m_k)$, $1<p<\infty$ in \cite{AS} using this Calder\'on--Zygmund condition in Dunkl setting.

Consider
$ \widetilde{\varphi}_{n,\varepsilon}$  a $C^\infty-$function on $\mathbb{R}$, such that:
\begin{itemize}
\item $\widetilde{\varphi}_{n,\varepsilon}$ is odd .
 \item $\widetilde{\varphi}_{n,\varepsilon}$ is supported in $\set{t\in \mathbb{R}; \;|t|\geq\varepsilon}$.
 \item $\widetilde{\varphi}_{n,\varepsilon}=1$ in $\set{t\in \mathbb{R}; \;t\geq\varepsilon+\frac{1}{n}}$.
 \item $\abs{\widetilde{\varphi}_{n,\varepsilon}}\leq1$.
 \end{itemize}
Let
$$\widetilde{\phi}_{n,\varepsilon}(t)=\int_{-\infty}^t \frac{\widetilde{\varphi}_{n,\varepsilon}(u)}{|u|^{p_k-1}} \;du\;\quad \mbox{and}
\quad \phi_{n,\varepsilon}(y)=\widetilde{\phi}_{n,\varepsilon}(\norm{ y }),\quad t\in\mathbb{R},\; y\in\mathbb{R}^N.$$ Clearly,
 $\phi_{n,\varepsilon}$  is a $C^\infty$ radial function and
   $$\lim_{n\rightarrow +\infty}\widetilde{\varphi}_{n,\varepsilon}(\norm{y})=1, \quad\forall \;y\in\mathbb{R}^N,\;\norm{y}>\varepsilon. $$
Denote by 
$\mathcal K_j^{(\varepsilon,n)}(x,y)=d_kT_j\tau_x(\phi_{n,\varepsilon})(-y)$. Here the action of $\tau_x$ on $\phi_{n,\varepsilon}$ is defined in the sense of distribution.
Then from the proof of \cite[Proposition 3.2]{AS}, 
$${\mathcal R}_j(f)(x)=\lim_{\varepsilon\rightarrow0}\lim_{n\rightarrow\infty}\int_{\mathbb{R}^N}\mathcal K_j^{(\varepsilon,n)}(x,y)f(y)dm_k(y),$$
where 
\begin{eqnarray*} T_j\tau_x(\phi_{n,\varepsilon})(-y)&=&
      \int_{\mathbb{R}^N}
\frac{(\eta_j-y_j)\widetilde{\varphi}_{n,\varepsilon}(A(x,y,\eta))}{  A^{p_k}(x,y,\eta) }d\mu_x(\eta)\\ &+&
\sum_{\alpha\in R_+}k(\alpha)\alpha_j\int_{\mathbb{R}^N}\frac{\widetilde{\phi}_{n,\varepsilon}(A(x,\sigma_\alpha.y,\eta))-
\widetilde{\phi}_{n,\varepsilon}(A(x,y,\eta))}
{<y,\alpha>}d\mu_x(\eta),
\end{eqnarray*}
and $\int_{\mathbb{R}^N}\mathcal K_j^{(\varepsilon,n)}(x,y)f(y)dm_k(y)$ could be an integrable singular integral for $f\in L^2(m_k)$ because $suppT_j\phi_{n,\varepsilon}\subseteq (B{(0,\varepsilon)})^c$ does not necessarily imply $suppT_j\tau_x(\phi_{n,\varepsilon})(-\cdot)$
$\subseteq(\bigcup_{g\in G}B(gx,\varepsilon))^c$ as is shown in Section 2.

For any $f\in L^2(m_k)$ with compact support, if $\mathcal R_j^\ast$ is the adjoint operator of
$\mathcal R_j$, then 

$$\mathcal R_j^\ast(f)(y)=\lim_{\varepsilon\rightarrow0}\lim_{n\rightarrow\infty}\int_{\mathbb{R}^N}\mathcal K_j^{(\varepsilon,n)}(x,y)f(x)dm_k(x).$$
By $\mathcal R_j=-\mathcal R_j^\ast$,
\begin{equation}{}\\\mathcal R_j(f)(y)=-\lim_{\varepsilon\rightarrow0}\lim_{n\rightarrow\infty}\int_{\mathbb{R}^N}\mathcal K_j^{(\varepsilon,n)}(x,y)f(x)dm_k(x).\end{equation}
If
$y\in\mathbb{R}^N$ satisfies $Gy\cap suppf=\varnothing$, then $d_G(x,y)>0$ for all $x\in suppf$.
For any $0<\varepsilon<d_G(x,y)$, from (3.4),
$$\varepsilon<  A(x,y,\eta)\; ; \quad \varepsilon<  A(x,\sigma_\alpha.y,\eta),\quad x\in supp(f),\;\eta\in co(G.x).$$
Then from the same argument as in the proof of \cite[Proposition 3.2]{AS}, 
$$\lim_{\varepsilon\rightarrow0}\lim_{n\rightarrow\infty}\mathcal K_j^{(\varepsilon,n)}(x,y)={\mathcal K}_j(x,y)$$
and
\begin{equation}{}\\\mathcal R_j(f)(y)=-\int_{\mathbb{R}^N}\mathcal K_j(x,y)f(x)dm_k(x), \;Gy\cap suppf=\varnothing.\end{equation}
with the aid of dominated convergence theorem.

\

\section{The Dunkl--type $BMO$ Space and Proof of Theorem 1.1}

The study of Dunkl--type $BMO$ space dates back to \cite {Gu}, where the space was defined for the one dimensional case. Here we will define the Dunkl--type $BMO$ space for multidimensional cases.

Given a function $f\in L_{loc}^1(m_k)$, and a ball $B(x,r)$. Denote $B_r\equiv B(0,r)$. Let $f_{B_r}(x)$
be the average of $\tau_xf$ on $B_r$: $$f_{B_r}(x)=\frac1{m_k(B_r)}\int_{B_r}\tau_xf(y)dm_k(y).$$
\begin{defn}
The Dunkl--type $BMO$ space is the space of all those
functions in $L_{loc}^1(m_k)$ satisfying
$\begin{array}{l}{\left\|f\right\|}_{\ast,\;k}<\infty\end{array}$, where $${\left\|f\right\|}_{\ast,\;k}=\underset{r>0,\;x\in\mathbb{R}^N}{sup}\frac1{m_k(B_r)}\int_{B_r}\vert\tau_xf(y)-f_{B_r}(x)\vert dm_k(y).$$ We can consider
$BMO$ as the quotient of the above space by the space of constant
functions to let ${\left\|\cdot\right\|}_{\ast,k}$ be a norm.
\end{defn}
\noindent{\it Proof of Theorem 1.1}.

Given a function $f$ in $L^\infty$ compactly supported, thanks to Theorem 1.1.\;i), $\tau_xf$ is compactly supported. Write $\tau_xf=g_1+g_2$, where $g_1=(\tau_xf)\chi_{B_{2r}}$, and $g_2=(\tau_xf)\chi_{(B_{2r})^c}$. This is the only way of decomposition because as is shown in Section 2, $suppf\subseteq\bigcup_{g\in G}B(gx,r)$ can not usually imply $supp\tau_xf\subseteq B(0,r)$. For any $y\in B_r$, $Gy\cap suppg_2=\varnothing$. Then by (3.2), (4.1), (4.2) and the uniform boundedness assumption of Dunkl translations,
\begin{align*}\left|\mathcal R_jg_2(y)-\mathcal R_jg_2(0)\right|&=\left|\int_{\mathbb{R}^N}(\mathcal K_j(z,\;y)-\mathcal K_j(z,\;0))g_2(z)dm_k(z)\right|\\&=\left|\int_{(B_{2r})^c}(\mathcal K_j(z,\;y)-\mathcal K_j(z,\;0))\tau_xf(z)dm_k(z)\right|\\&\leq\int_{d_G(0,z)>2\left\|y\right\|}\left|\mathcal K_j(z,\;y)-\mathcal K_j(z,\;0)\right|dm_k(z){\left\|\tau_xf\right\|}_\infty\\&\leq C{\left\|f\right\|}_\infty\end{align*}

Then using again the uniform boundedness assumption of Dunkl translations, and by the $L^2$ boundedness of the Riesz transform and \eqref{mk},
\begin{align*}\frac1{m_k(B_r)}\int_{B_r}\left|\mathcal R_jg_1\right|&\leq\left(\frac1{m_k(B_r)}\int_{B_r}\left|\mathcal R_jg_1\right|^2\right)^{\textstyle\frac12}\\&\leq\left(\frac1{m_k(B_r)}\int\left|\left(\tau_xf\right)\chi_{B_{2r}}\right|^2\right)^{\textstyle\frac12}\\&=\left(\frac1{m_k(B_r)}\int_{B_{2r}}\left|\tau_xf\right|^2\right)^{\textstyle\frac12}\\&\leq C{\left\|\tau_xf\right\|}_\infty\\&\leq C{\left\|f\right\|}_\infty.\end{align*}
Therefore,
\begin{align*}\frac1{m_k(B_r)}\int_{B_r}\left|\mathcal R_j\tau_xf(y)-\mathcal R_jg_2(0)\right|dm_k(y)\leq\frac1{m_k(B_r)}\int_{B_r}\left|\mathcal R_jg_1(y)\right|dm_k(y)\\\;\;+\frac1{m_k(B_r)}\int_{B_r}\left|\mathcal R_jg_2(y)-\mathcal R_jg_2(0)\right|dm_k(y)\leq C{\left\|f\right\|}_\infty.\end{align*}

We will then extend the definition of Riesz transforms for Dunkl transform to all of $L^\infty$. For any function $f\in L^\infty$, define
\begin{align}\label{Rj}{\mathcal R}_j(f)(y)=-\lim_{\varepsilon\rightarrow0}\lim_{n\rightarrow\infty}\int_{\mathbb{R}^N}\left(\mathcal K_j^{(\varepsilon,n)}(z,y)-\mathcal K_j^{(\varepsilon,n)}(z,0)\right)f(z)dm_k(z).\end{align}
For any $\varepsilon>0$ and any $y\in\mathbb{R}^N$, there exists a sufficiently large $r>2\varepsilon$ such that $y\in B_r$.
Write $f=f_1+f_2$, where $f_1=f\chi_{B_{2r}}$, and $f_2=f\chi_{(B_{2r})^c}$.
Then $f_1$ belongs to $L^2(m_k)$ and so $\mathcal R_jf_1(y)$ converges almost everywhere.
For any natural number $n$ larger than $\frac1\varepsilon$ and any $z\in suppf_2$,
$$d_G\left(z,0\right)\geq2r>\varepsilon+\frac1n;\;\;\;\;d_G\left(z,y\right)\geq d_G\left(z,0\right)-\left\|y\right\|\geq r>\varepsilon+\frac1n.$$
So for the above $\varepsilon$ and $n$, $$A(z,0,\eta)>\varepsilon+\frac1n;\quad A(z,y,\eta)>\varepsilon+\frac1n,\quad\eta\in co(G.x),$$
and $$\mathcal K_j^{(\varepsilon,n)}(z,0)={\mathcal K}_j(z,0);\quad\mathcal K_j^{(\varepsilon,n)}(z,y)={\mathcal K}_j(z,y).$$
Therefore,
$${\mathcal R}_j(f_2)(y)=-\int_{{(B_{2r})}^c}\left({\mathcal K}_j(z,y)-{\mathcal K}_j(z,0)\right)f(z)dm_k(z).$$
This integral converges since
\begin{align*}&\;\;\left|\int_{(B_{2r})^c}(\mathcal K_j(z,\;y)-\mathcal K_j(z,\;0))f(z)dm_k(z)\right|\\&\leq\int_{d_G(0,z)>2\left\|y\right\|}\left|\mathcal K_j(z,\;y)-\mathcal K_j(z,\;0)\right|dm_k(z){\left\|f\right\|}_\infty\leq C{\left\|f\right\|}_\infty,\;y \in B_r.\end{align*}
So the above definition \eqref{Rj} for Riesz transforms for Dunkl transform on $L^\infty$ makes sense for any $y\in\mathbb{R}^N$and coincides with formula (4.2) for compactly supported functions on  $L^2(m_k)$ as Dunkl--type $BMO$ functions since the two formulae differ by a constant.

Under the uniform boundedness assumption of Dunkl translations, for any $x\in \mathbb R^n$, $r>0$ and all $f \in L^{\infty}$, $\tau_xf\in L^\infty$. Write $\tau_xf=g_1+g_2$, where $g_1=(\tau_xf)\chi_{B_{2r}}$, and $g_2=(\tau_xf)\chi_{(B_{2r})^c}$. Then 
$${\mathcal R}_j(g_1)(y)=-\lim_{\varepsilon\rightarrow0}\lim_{n\rightarrow\infty}\int_{\mathbb{R}^N}\mathcal K_j^{(\varepsilon,n)}(z,y)g_1(z)dm_k(z)+{\mathcal R}_j(g_1)(0),$$
and for any $y\in B_r$,
$${\mathcal R}_j(g_2)(y)=-\int_{{(B_{2r})}^c}\left({\mathcal K}_j(z,y)-{\mathcal K}_j(z,0)\right)\tau_xf(z)dm_k(z).$$
By the same argument as for compactly supported functions, we have
\begin{align*}&\frac1{m_k(B_r)}\int_{B_r}\left|\mathcal R_j\tau_xf(y)-{\mathcal R}_j(g_1)(0)\right|dm_k(y)=\frac1{m_k(B_r)}\int_{B_r}\Big|\mathcal R_j\left\{\left(\tau_xf\right)\chi_{B_{2r}}\right\}\left(y\right)\\&+\int_{(B_{2r})^c}(-\mathcal K_j(z,\;y)+\mathcal K_j(z,\;0))\tau_xf(z)dm_k(z)\Big|dm_k(y)\leq C{\left\|f\right\|}_\infty.\end{align*}

The following Lemma will then imply the boundedness of Riesz transforms for Dunkl transform from $L^\infty$ to Dunkl--type $BMO$ space.

\begin{lem} Under the uniform boundedness assumption of Dunkl translations, for any $f\in L^\infty$ and any fixed $x\in\mathbb{R}^N$, $r>0$, $\mathcal R_j\tau_xf(y)$ and $\tau_x\mathcal R_jf(y)$ differ by a constant independent of $y$ for $y\in B_r$.
\end{lem}

\begin{proof}

This statement is obvious for functions compactly supported on $L^\infty$, implying that for any function $f\in L^\infty$ compactly supported,
\begin{align}\frac1{m_k(B_r)}\int_{B_r}\left|\left[\mathcal R_j,\tau_x\right]f(y)-\frac1{m_k(B_r)}\int_{B_r}\left[\mathcal R_j,\tau_x\right]f\right|dm_k(y)=0,\end{align}
where $\left[X,Y\right]:=XY-YX.$

  For all $f\in L^\infty$, if $\left\{D_i\right\}$ is a countable open cover of $\mathbb{R}^N$, where each $D_i$ is bounded, then by partition of unity, $f$ can be written as $\textstyle\overset\infty{\underset{i=1}{\sum f_i}}$, where $suppf_i\subset D_i$. Denote $g:=\left[\mathcal R_j,\tau_x\right]f$. Then 
$$g=\sum_{i=1}^\infty g^{(i)},\;g^{(i)}=\left[\mathcal R_j,\tau_x\right]f_i,$$
and by (5.1),
\begin{align*}&\frac1{m_k(B_r)}\int_{B_r}\left|g(y)-\frac1{m_k(B_r)}\int_{B_r}g\right|dm_k(y)\\=&\frac1{m_k(B_r)}\int_{B_r}\left|\sum_{i=1}^\infty g^{(i)}(y)-\sum_{i=1}^\infty \frac1{m_k(B_r)}\int_{B_r}g^{(i)}\right|dm_k(y)\\\leq&\sum_{i=1}^\infty\frac1{m_k(B_r)}\int_{B_r}\left|g^{(i)}(y)-\frac1{m_k(B_r)}\int_{B_r}g^{(i)}\right|dm_k(y)\\&=0.\end{align*}
And so  $g(y)=\left[\mathcal R_j,\tau_x\right]f(y)$ is a constant independent of $y$ on $y\in B_r$.
\qedhere
\end{proof}

Denote by $C_x$ the constant $\tau_x\mathcal R_jf(y)$ and $\mathcal R_j\tau_xf(y)$ differ plus ${\mathcal R}_j(g_1)(0)$. Then 
\begin{align*}\frac1{m_k(B_r)}\int_{B_r}\left|\tau_x\mathcal R_jf(y)-{\left(\mathcal R_jf\right)}_{B_r}(x)\right|dm_k(y)\leq\frac1{m_k(B_r)}\int_{B_r}\left|\tau_x\mathcal R_jf(y)-C_x\right|dm_k(y)\\+\left|C_x-{\left(\mathcal R_jf\right)}_{B_r}(x)\right|\leq C{\left\|f\right\|}_\infty+\frac1{m_k(B_r)}\int_{B_r}\left|\tau_x\mathcal R_jf(y)-C_x\right|dm_k(y)\leq2C{\left\|f\right\|}_\infty.\end{align*}
$\hfill\Box$

\section*{Acknowledgments}The author would like to thank Margit R\"osler very much for her correction of Theorem 1.2 ii) and some valuable comments, and thank the reviewer and his former adviser Heping Wang for valuable suggestions. The paper is based on the master thesis of the author at
Capital Normal University

\


\begin{thebibliography}{99}


\bibitem{AA} B. Amri, JP. Anker, M. Sifi, Three results in Dunkl analysis, In Colloq. Math Vol. 118(2010), no. 1, 299--312.

\bibitem{AS} B. Amri, M. Sifi, Riesz transforms for the Dunkl transform, Ann. Math.Blaise Pascal 19(2012), no. 1, 247--262.

\bibitem{ADH}  JP. Anker, J. Dziuba\'nski, A. Hejna, Harmonic functions, conjugate harmonic functions and the Hardy space $H^1$ in the rational Dunkl setting, Journal of Fourier Analysis and Applications, 25(5)(2019), 2356--2418


\bibitem{DW} F. Dai, H. Wang, A transference theorem for the Dunkl transform and its applications, J. Funct.
Anal. 258.12(2010), 4052--4074.


\bibitem{Du} C.F. Dunkl, Differential-difference operators associated to reflection groups, Trans. Amer. Math. Soc. 311(1989), no. 1, 167--183.

\bibitem{DH} J. Dziuba\'nski, A. Hejna, Remark on atomic decompositions for Hardy space $H^1$ in the rational Dunkl setting, Studia Mathematica, 251(2020), 89--110..

\bibitem{DH1} J. Dziuba\'nski, A. Hejna, H\"ormander's multiplier theorem for the Dunkl transform, J. Funct.Anal(2019).

\bibitem{I} V.I. Ivanov, Weighted inequalities for Dunkl--Riesz transforms and Dunkl gradient. Chebyshevskii Sbornik. 21(4)(2020), 97--106. (In Russ.).

\bibitem {Ga} L. Gallardo, C. Rejeb, Support properties of the intertwining and the mean value operators in Dunkl theory. Proceedings of the American Mathematical
Society 146.1(2017), 1.

\bibitem {Ga1} L. Gallardo, C. Rejeb, A new mean value property for harmonic functions
relative to the Dunkl-Laplacian operator and applications, Trans. Amer. Math. Soc. 368
(2016), no. 5, 3727–3753, DOI 10.1090/tran/6671.

\bibitem {Go} D.V. Gorbachev, V.I. Ivanov, S.Y. Tikhonov, Positive $L^p$--Bounded Dunkl-Type Generalized Translation Operator and Its Applications. Constr Approx 49(2019), 555--605.

\bibitem {Gu} V.S. Guliyev, Y.Y. Mammadov, On fractional maximal function and fractional integrals associated with the Dunkl operator on the real line[J]. Journal of
Mathematical Analysis and Applications 353.1(2009), 449--459.


\bibitem{R} M. R\"osler, A positive radial product formula for the Dunkl kernel, Trans. Amer.Math. Soc. 355(2003), no. 6, 2413--2438.

\bibitem{TX} S. Thangavelu, Y. Xu, Riesz transform and Riesz potentials for Dunkl transform, J. Comput. Appl. Math. 199.1(2007), 181--195.


\end{thebibliography}
\end{document}